\documentclass[11pt]{article}%
\usepackage{amsfonts}
\usepackage{amsmath}
\usepackage{amssymb}
\usepackage{graphicx}%
\providecommand{\U}[1]{\protect\rule{.1in}{.1in}}
\newtheorem{theorem}{Theorem}
\newtheorem{acknowledgement}[theorem]{Acknowledgement}

\newtheorem{corollary}[theorem]{Corollary}

\newtheorem{proposition}[theorem]{Proposition}
\newtheorem{remark}[theorem]{Remark}

\newenvironment{proof}[1][Proof]{\noindent\textbf{#1.} }{\ \rule{0.5em}{0.5em}}
\begin{document}

\title{Symplectic and Hamiltonian Deformations of Gabor Frames}
\author{Maurice A. de Gosson\\University of Vienna, Faculty of Mathematics, NuHAG}
\maketitle

\begin{abstract}
We study symplectic deformations of Gabor frames, using the covariance
properties of the Heisenberg operators. This allows us to recover in a very
simple way known results. We thereafter propose a general deformation scheme
by Hamiltonian isotopies, which are paths of Hamiltonian flows. We define and
study in detail a weak notion of Hamiltonian deformations, using ideas from
semiclassical analysis due to Heller and Hagedorn. This method can be easily
implemented using symplectic integrators.

\end{abstract}

\section{Introduction}

The theory of Gabor frames (or Weyl--Heisenberg frames as they are also
called) is a rich and expanding topic of harmonic analysis. It has many
applications in time-frequency analysis, signal theory, and mathematical
physics. The aim of this article is to initiate a systematic study of the
symplectic transformation properties of Gabor frames, both in the linear and
nonlinear cases. Strangely enough, the use of symplectic techniques in the
theory of Gabor frames is very often ignored; one example (among many others)
being Casazza's seminal paper \cite{casa} on modern tools for Weyl--Heisenberg
frame theory, where the word \textquotedblleft symplectic\textquotedblright%
\ does not appear a single time in the 127 pages of this paper! There are
however exceptions: in Gr\"{o}chenig's treatise \cite{Gro} the metaplectic
representation is used to study various symmetries; the same applies to the
recent paper by Pfander \textit{et al.} \cite{goetz}, elaborating on earlier
work \cite{han} by Han and Wang, where symplectic transformations are
exploited to study various properties of Gabor frames.

In this paper we consider the notion of deformation of a Gabor system using
the \textit{Hamiltonian isotopies }we introduce in section \ref{sechamn}. A
Hamiltonian isotopy is a curve $(f_{t})_{0\leq t\leq1}$ of diffeomorphisms of
phase space $\mathbb{R}^{2n}$ starting at the identity, and such that there
exists a Hamiltonian function $H$, usually time-dependent such that the
(generalized) phase flow $(f_{t}^{H})_{t}$ determined by the Hamilton
equations%
\begin{equation}
\dot{x}=\partial_{p}H(x,p,t)\text{ \ , \ }\dot{p}=-\partial_{p}H(x,p,t)
\label{ham}%
\end{equation}
consists of precisely the mappings $f_{t}$ for $0\leq t\leq1$. It follows that
a Hamiltonian isotopy consists of symplectomorphisms (or canonical
transformations, as they are called in Physics). Given a Gabor system
${\mathcal{G}}(\phi,\Lambda)$ with window (or atom) $\phi$ and lattice
$\Lambda$ we want to find a working definition of the deformation of
${\mathcal{G}}(\phi,\Lambda)$ by a Hamiltonian isotopy $(f_{t})_{0\leq t\leq
1}$. While it is clear that the deformed lattice should be the image
$\Lambda_{t}=f_{t}(\Lambda)$ of the original lattice $\Lambda$, it is less
clear what the deformation $\phi_{t}=f_{t}(\phi)$ of the window $\phi$ should
be. A clue is given by the linear case: assume that the mappings $f_{t}$ are
linear, \textit{i.e.} symplectic matrices $S_{t}$; assume in addition that
there exists an infinitesimal symplectic transformation $X$ such that
$S_{t}=e^{tX}$ for $0\leq t\leq1$. Then $(S_{t})_{t}$ is the flow determined
by the Hamiltonian function%
\begin{equation}
H(x,p)=-\tfrac{1}{2}(x,p)^{T}JX(x,p) \label{hxp}%
\end{equation}
where $J$ is the standard symplectic matrix. There exists a one-parameter
group of unitary operators $(\widehat{S}_{t})_{t}$ satisfying the operator
Schr\"{o}dinger equation%
\[
i\hbar\frac{\mathrm{d}}{\mathrm{d}t}\widehat{S}_{t}=H(x,-i\hbar\partial
_{x})\widehat{S}_{t}%
\]
where the formally self-adjoint operator $H(x,-i\hbar\partial_{x})$ is
obtained by replacing formally $p$ with $-i\hbar\partial_{x}$ in (\ref{hxp});
the matrices $S_{t}$ and the operators $\widehat{S}_{t}$ correspond to each
other via the metaplectic representation of the symplectic group. This
suggests that we define the deformation of the initial window $\phi$ by
$\phi_{t}=\widehat{S}_{t}\phi$. It turns out that this definition is
satisfactory, because it allows to recover, setting $t=1$, known results on
the image of Gabor frames by linear symplectic transformations. This example
is thus a good guideline; however one encounters difficulties as soon as one
want to extend it to more general situations. While it is \textquotedblleft
reasonably\textquotedblright\ easy to see what one should do when the
Hamiltonian isotopy consists of an arbitrary path of symplectic matrices (this
will be done in section \ref{secsymp1}), it is not clear at all what a
\textquotedblleft good\textquotedblright\ definition should be in the general
nonlinear case: this is discussed in section \ref{sechade}, where we suggest
that a natural choice would be to extend the linear case by requiring that
$\phi_{t}$ should be the solution of the Schr\"{o}dinger equation
\[
i\hbar\frac{\mathrm{d}}{\mathrm{d}t}\phi_{t}=\widehat{H}\phi_{t}%
\]
associated with the Hamiltonian function $H$ determined by the equality
$(f_{t})_{0\leq t\leq1}=(f_{t}^{H})_{0\leq t\leq1}$; the Hamiltonian operator
$\widehat{H}$ would then be associated with the function $H$ by using, for
instance, the Weyl correspondence. Since the method seems to be difficult to
study theoretically and to implement numerically, we propose what we call a
notion of \emph{weak deformation}, where the exact definition of the
transformation $\phi\longmapsto\phi_{t}$ of the window $\phi$ is replaced with
a correspondence used in semiclassical mechanics, and which consists in
propagating the \textquotedblleft center\textquotedblright\ of a sufficiently
sharply peaked initial window $\phi$ (for instance a coherent state, or a more
general Gaussian) along the Hamiltonian trajectory. This definition coincides
with the definition already given in the linear case, and has the advantage of
being easily computable using the method of symplectic integrators (which we
review in section \ref{secsympalg}) since all what is needed is the knowledge
of the phase flow determined by a certain Hamiltonian function. Finally we
discuss possible extensions of our method.

\subsubsection*{Notation and terminology}

The generic point of the phase space $\mathbb{R}^{2n}\equiv\mathbb{R}%
^{n}\times\mathbb{R}^{n}$ is denoted by $z=(x,p)$ where we have set
$x=(x_{1},...,x_{n})$, $p=(p_{1},...,p_{n})$. The scalar product of two
vectors, say $p$ and $x$, is denoted by $p\cdot x$ or simply $px$. When matrix
calculations are performed, $z,x,p$ are viewed as column vectors. We will
equip $\mathbb{R}^{2n}$ with the standard symplectic structure
\[
\sigma(z,z^{\prime})=p\cdot x^{\prime}-p^{\prime}\cdot x;
\]
in matrix notation $\sigma(z,z^{\prime})=(z^{\prime})^{T}Jz$ where $J=%
\begin{pmatrix}
0 & I\\
-I & 0
\end{pmatrix}
$ ($0$ and $I$ are the $n\times n$ zero and identity matrices). The symplectic
group of $\mathbb{R}^{2n}$ is denoted by $\operatorname*{Sp}(n)$; it consists
of all linear automorphisms of $\mathbb{R}^{2n}$ such that $\sigma
(Sz,Sz^{\prime})=\sigma(z,z^{\prime})$ for all $z,z^{\prime}\in\mathbb{R}%
^{2n}$. Working in the canonical basis $\operatorname*{Sp}(n)$ is identified
with the group of all real $2n\times2n$ matrices $S$ such that $S^{T}JS=J$
(or, equivalently, $SJS^{T}=J$).

We will write $\mathrm{d}z=\mathrm{d}x\mathrm{d}p$ where $\mathrm{d}%
x=\mathrm{d}x_{1}\cdot\cdot\cdot\mathrm{d}x_{n}$ and $\mathrm{d}%
p=\mathrm{d}p_{1}\cdot\cdot\cdot\mathrm{d}p_{n}$. The scalar product on
$L^{2}(\mathbb{R}^{n})$ is defined by%
\[
(\psi|\phi)=\int_{\mathbb{R}^{n}}\psi(x)\overline{\phi(x)}\mathrm{d}x
\]
and the associated norm is denoted by $||\cdot||$. The Schwartz space of
rapidly decreasing functions is denoted by $\mathcal{S}(\mathbb{R}^{n})$ and
its dual (the space of tempered distributions) by $\mathcal{S}^{\prime
}(\mathbb{R}^{n})$.

\section{Gabor Frames}

Gabor frames are a generalization of the usual notion of basis; see for
instance Gr\"{o}chenig \cite{Gro}, Feichtinger and Gr\"{o}chenig \cite{fegro},
Balan \textit{et al.} \cite{balan}, Heil \cite{Heil}, Casazza \cite{casa} for
a detailed treatment of this topic. In what follows we give a slightly
modified version of the usual definition, better adapted to the study of
symplectic symmetries.

\subsection{Definition}

Let $\phi$ be a non-zero square integrable function (hereafter called
\textit{window}) on ${\mathbb{R}}^{n}$, and a lattice $\Lambda$ in
${\mathbb{R}}^{2n}$, \textit{i.e.} a discrete subset of ${\mathbb{R}}^{2n}$.
The associated $\hbar$-Gabor system is the set of square-integrable functions
\[
{\mathcal{G}}(\phi,\Lambda)=\{\widehat{T}^{\hbar}(z)\phi:z\in\Lambda\}
\]
where $\widehat{T}^{\hbar}(z)=e^{-i\sigma(\hat{z},z)/\hbar}$ is the Heisenberg
operator. The action of this operator is explicitly given by the formula
\begin{equation}
\widehat{T}^{\hbar}(z_{0})\phi(x)=e^{i(p_{0}x-p_{0}x_{0}/2)/\hbar}\phi
(x-x_{0}) \label{heiwe}%
\end{equation}
(see \textit{e.g.} \cite{Birk,Birkbis,Littlejohn}; it will be justified in
section \ref{sectrans}). We will call the Gabor system ${\mathcal{G}%
}(g,\Lambda)$ a $\hbar$-frame for $L^{2}({\mathbb{R}}^{n})$, if there exist
constants $a,b>0$ (the \textit{frame bounds}) such that
\begin{equation}
a||\psi||^{2}\leq\sum_{z_{0}\in\Lambda}|(\psi|\widehat{T}^{\hbar}(z_{0}%
)\phi)|^{2}\leq b||\psi||^{2} \label{frame1}%
\end{equation}
for every square integrable function $\psi$ on ${\mathbb{R}}^{n}$. When $a=b $
the $\hbar$-frame ${\mathcal{G}}(g,\Lambda)$ is said to be \emph{tight}.

\begin{remark}
The product $(\psi|\widehat{T}^{\hbar}(z_{0})\phi)$ is, up to the factor
$(2\pi\hbar)^{-n}$, Woodward's cross-ambiguity function \cite{wood}; it is up
to a (symplectic Fourier transform) the cross-Wigner distribution $W(\psi
,\phi)$ as was already observed by Klauder \cite{Klauder}; see
\cite{Folland,Birk,Birkbis}.
\end{remark}

\subsection{Rescaling properties}

For the choice $\hbar=1/2\pi$ the notion of $\hbar$-Gabor frame coincides with
the usual notion of Gabor frame as found in the literature. In fact, in this
case, writing $\widehat{T}(z)=\widehat{T}^{1/2\pi}(z)$ and $p=\omega$, we have%
\[
|(\psi|\widehat{T}(z)\phi)|=|(\psi|\tau(z)\phi)|
\]
where $\tau(z)$ is the modulation operator defined by%
\[
\tau(z_{0})\phi(x)=e^{2\pi i\omega_{0}x}\phi(x-x_{0})
\]
for $z_{0}=(x_{0},\omega_{0})$. The two following elementary results can be
used to go from one definition to the other:

\begin{proposition}
\label{Propres1}Let $D^{\hbar}=%
\begin{pmatrix}
I & 0\\
0 & 2\pi\hbar I
\end{pmatrix}
$. The system $\mathcal{G}(\phi,\Lambda)$ is a Gabor frame if and only if
$\mathcal{G}(\phi,D^{\hbar}\Lambda)$ is a $\hbar$-Gabor frame.
\end{proposition}

\begin{proof}
We have $\widehat{T}^{\hbar}(x_{0},2\pi\hbar p_{0})=\widehat{T}(x_{0},p_{0})$
where $\widehat{T}(x_{0},p_{0})=\widehat{T}^{1/2\pi}(x_{0},p_{0})$. By
definition $\mathcal{G}(\phi,\Lambda)$ is a Gabor frame if and only if%
\[
a||\psi||^{2}\leq\sum_{z_{0}\in\Lambda}|(\psi|\widehat{T}(z_{0})\phi)|^{2}\leq
b||\psi||^{2}%
\]
for every $\psi\in L^{2}(\mathbb{R}^{n})$ that is%
\[
a||\psi||^{2}\leq\sum_{(x_{0},p_{0})\in\Lambda}|(\psi|\widehat{T}(x_{0}%
,p_{0})\phi)|^{2}\leq b||\psi||^{2};
\]
this inequality is equivalent to%
\[
a||\psi||^{2}\leq\sum_{(x_{0},p_{0})\in\Lambda}|(\psi|\widehat{T}^{\hbar
}(x_{0},2\pi\hbar p_{0})\phi)|^{2}\leq b||\psi||^{2}%
\]
that is to%
\[
a||\psi||^{2}\leq\sum_{(x_{0},(2\pi\hbar)^{-1}p_{0})\in\Lambda}|(\psi
|\widehat{T}^{\hbar}(x_{0},p_{0})\phi)|^{2}\leq b||\psi||^{2}%
\]
hence the result since $(x_{0},(2\pi\hbar)^{-1}p_{0})\in\Lambda$ means that
$(x_{0},p_{0})\in D^{\hbar}\Lambda$.
\end{proof}

We can also rescale simultaneously the lattice and the window
(\textquotedblleft change of Planck's constant\textquotedblright):

\begin{proposition}
\label{Propres2}Let $\mathcal{G}(\phi,\Lambda)$ be a Gabor system, and set
\begin{equation}
\phi^{\hbar}(x)=(2\pi\hbar)^{-n/2}\phi(x/\sqrt{2\pi\hbar}). \label{fih}%
\end{equation}
Then $\mathcal{G}(\phi,\Lambda)$ is a frame if and only if $\mathcal{G}%
(\phi^{\hbar},\sqrt{2\pi\hbar}\Lambda)$ is a $\hbar$-frame.
\end{proposition}

\begin{proof}
We have $\phi^{\hbar}=\widehat{M}_{1/\sqrt{2\pi\hbar}I,0}\phi$ where
$\widehat{M}_{1/\sqrt{2\pi\hbar}I,0}\in\operatorname*{Mp}(n)$ has projection
\[
M_{1/\sqrt{2\pi\hbar}}=%
\begin{pmatrix}
(2\pi\hbar)^{1/2}I & 0\\
0 & (2\pi\hbar)^{-1/2}I
\end{pmatrix}
\]
on $\operatorname*{Sp}(n)$ (see Appendix A). The Gabor system $\mathcal{G}%
(\phi^{\hbar},\sqrt{2\pi\hbar}\Lambda)$ is a $\hbar$-frame if and only%
\[
a||\psi||^{2}\leq\sum_{z_{0}\in\sqrt{2\pi\hbar}\Lambda}|(\psi|\widehat
{T}(z_{0})\widehat{M}_{1/\sqrt{2\pi\hbar}I,0}\phi)|^{2}\leq b||\psi||^{2}%
\]
for every $\psi\in L^{2}(\mathbb{R}^{n})$, that is, taking the symplectic
covariance formula (\ref{symco}) into account, if and only if
\[
a||\psi||^{2}\leq\sum_{z_{0}\in\sqrt{2\pi\hbar}\Lambda}|(\widehat{M}%
_{\sqrt{2\pi\hbar}I,0}\psi|\widehat{T}((2\pi\hbar)^{-1/2}x_{0},(2\pi
\hbar)^{1/2}p_{0})\phi)|^{2}\leq b||\psi||^{2}.
\]
But this is inequality is equivalent to%
\[
a||\psi||^{2}\leq\sum_{z_{0}\in D^{\hbar}\Lambda}|(\psi|\widehat{T}(z_{0}%
)\phi)|^{2}\leq b||\psi||^{2}%
\]
and one concludes using Proposition \ref{Propres1}.
\end{proof}

\begin{remark}
In Appendix A we state a rescaling property for the covering projection
$\pi^{\hbar}:$ $\widehat{S}\longmapsto S$ of metaplectic group
$\operatorname*{Mp}(n)$ onto $\operatorname*{Sp}(n)$ (formula (\ref{pich})).
\end{remark}

\section{Symplectic Covariance}

The following formula will play a fundamental role in our study of symplectic
covariance properties of frames. It relates Heisenberg--Weyl operators, linear
symplectic transformations, and metaplectic operators (we refer to Appendix A
for a concise review of the metaplectic group $\operatorname*{Mp}(n)$ and its properties).

Let $\widehat{S}\in\operatorname*{Mp}(n)$ have projection $\pi^{\hbar
}(\widehat{S})=S\in\operatorname*{Sp}(n)$. Then%

\begin{equation}
\widehat{T}^{\hbar}(z)\widehat{S}=\widehat{S}\widehat{T}^{\hbar}(S^{-1}z).
\label{symco}%
\end{equation}

For a proof see \textit{e.g.} \cite{Birk,Birkbis,Littlejohn}; one easy way is
to prove this formula separately for each generator $\widehat{J}$,
$\widehat{M}_{L,m}$, $\widehat{V}_{P}$ of the metaplectic group.

\subsection{A first covariance result}

Gabor frames behave well under symplectic transformations of the lattice (or,
equivalently, under metaplectic transformations of the window). Let
$\widehat{S}\in\operatorname*{Mp}(n)$ have projection $S\in\operatorname*{Sp}%
(n)$. The following result is well-known, and appears in many places in the
literature (see \textit{e.g.} Gr\"{o}chenig \cite{Gro}, Pfander \textit{et
al.} \cite{goetz}). Our proof is somewhat simpler since it exploits the
symplectic covariance property of the Heisenberg--Weyl operators.

\begin{proposition}
\label{prop2}Let $\phi\in L^{2}(\mathbb{R}^{n})$ (or $\phi\in\mathcal{S}%
(\mathbb{R}^{n})$). A Gabor system $\mathcal{G}(\phi,\Lambda)$ is a $\hbar
$-frame if and only if $\mathcal{G}(\widehat{S}\phi,S\Lambda)$ is a $\hbar
$-frame; when this is the case both frames have the same bounds. In
particular, $\mathcal{G}(\phi,\Lambda)$ is a tight $\hbar$-frame if and only
if $\mathcal{G}(\phi,\Lambda)$ is.
\end{proposition}

\begin{proof}
Using the symplectic covariance formula (\ref{symco}) we have%
\begin{align*}
\sum_{z\in S\Lambda}|(\psi|\widehat{T}^{\hbar}(z)\widehat{S}\phi)|^{2}  &
=\sum_{z\in S\Lambda}|(\psi|\widehat{S}\widehat{T}^{\hbar}(S^{-1}z)\phi
)|^{2}\\
&  =\sum_{z\in\Lambda}|(\widehat{S}^{-1}\psi|\widehat{T}^{\hbar}(z)\phi)|^{2}%
\end{align*}
and hence, since $\mathcal{G}(\phi,\Lambda)$ is a $\hbar$-frame,
\[
a||\widehat{S}^{-1}\psi||^{2}\leq\sum_{z\in S\Lambda}|(\psi|\widehat{T}%
^{\hbar}(z)\widehat{S}\phi)|^{2}\leq b||\widehat{S}^{-1}\psi||^{2}.
\]
The result follows since $||\widehat{S}^{-1}\psi||=||\psi||$ because
metaplectic operators are unitary; the case $\phi\in\mathcal{S}(\mathbb{R}%
^{n})$ is similar since metaplectic operators are linear automorphisms of
$\mathcal{S}(\mathbb{R}^{n})$.
\end{proof}

\begin{remark}
The result above still holds when one assumes that the window $\phi$ belongs
to the Feichtinger algebra $S_{0}(\mathbb{R}^{n})$ (see Appendix B and the
discussion at the end of the paper).
\end{remark}

\subsection{Gaussian frames}

The problem of constructing Gabor frames $\mathcal{G}(\phi,\Lambda)$ in
$L^{2}(\mathbb{R})$ with an arbitrary window $\phi$ and lattice $\Lambda$ is
difficult and has been tackled by many authors (see for instance the comments
in \cite{GroLyu}, also \cite{goetz}). Very little is known about the existence
of frames in the general case. We however have the following characterization
of Gaussian frames which extends a classical result of Lyubarskii \cite{Lyu}
and Seip and Wallst\'{e}n \cite{sewa92}:

\begin{proposition}
\label{prop3}Let $\phi_{0}^{\hbar}(x)=(\pi\hbar)^{-n/4}e^{-|x|^{2}/2\hbar}$
(the standard centered Gaussian) and $\Lambda_{\alpha\beta}=\alpha
\mathbb{Z}^{n}\times\beta\mathbb{Z}^{n}$ with $\alpha=(\alpha_{1}%
,...,\alpha_{n}\mathbb{)}$ and $\beta=(\beta_{1},...,\beta_{n}\mathbb{)}$.
Then $\mathcal{G}(\phi_{0}^{\hbar},\Lambda_{\alpha\beta})$ is a frame if and
only if $\alpha_{j}\beta_{j}<2\pi\hbar$ for $1\leq j\leq n$.
\end{proposition}

\begin{proof}
Bourouihiya \cite{AB2} proves this for $\hbar=1/2\pi$; the result for
arbitrary $\hbar>0$ follows using Proposition \ref{Propres2}.
\end{proof}

It turns out that using the result above one can construct infinitely many
symplectic Gaussian frames using the theory of metaplectic operators:

\begin{proposition}
\label{prop4}Let $\phi_{0}^{\hbar}$ be the standard Gaussian. The Gabor system
$\mathcal{G}(\phi_{0}^{\hbar},\Lambda_{\alpha\beta})$ is a frame if and only
if $\mathcal{G}(\widehat{S}\phi_{0}^{\hbar},S\Lambda_{\alpha\beta})$ is a
frame (with same bounds) for every $\widehat{S}\in\operatorname*{Mp}(n)$.
Writing $S$ in block-matrix form $%
\begin{pmatrix}
A & B\\
C & D
\end{pmatrix}
$ the window $\widehat{S}\phi_{0}^{\hbar}$ is the Gaussian%
\begin{equation}
\widehat{S}\phi_{0}^{\hbar}(x)=\left(  \tfrac{1}{\pi\hbar}\right)  ^{n/4}(\det
X)^{1/4}e^{-\frac{1}{2\hbar}(X+iY)x\cdot x} \label{gauss}%
\end{equation}
where%
\begin{align}
X  &  =-(CA^{T}+DB^{T})(AA^{T}+BB^{T})^{-1}\label{xa11}\\
Y  &  =(AA^{T}+BB^{T})^{-1} \label{ya11}%
\end{align}
are symmetric matrices, and $X>0$.
\end{proposition}

\begin{proof}
That $\mathcal{G}(\phi_{0}^{\hbar},\Lambda_{\alpha\beta})$ is a frame if and
only if $\mathcal{G}(\widehat{S}\phi_{0}^{\hbar},S\Lambda_{\alpha\beta})$ is a
frame follows from Proposition \ref{prop2}. To calculate $\widehat{S}\phi
_{0}^{\hbar}$ it suffices to apply formulas (\ref{xaa}) and (\ref{xbb}) in
Appendix A.
\end{proof}

\subsection{An example\label{subsecex}}

Let us choose $\hbar=1/2\pi$ and consider the rotations%
\begin{equation}
S_{t}=%
\begin{pmatrix}
\cos t & \sin t\\
-\sin t & \cos t
\end{pmatrix}
\label{rot1}%
\end{equation}
(we assume $n=1$). The matrices $(S_{t})$ form a one-parameter subgroup of the
symplectic group $\operatorname*{Sp}(1)$. To $(S_{t})$ corresponds a unique
one-parameter subgroup $(\widehat{S}_{t})$ of the metaplectic group
$\operatorname*{Mp}(1)$ such that $S_{t}=\pi^{1/2\pi}(\widehat{S}_{t})$. It
follows from formula (\ref{swm}) in Appendix A that $\widehat{S}_{t}\phi$ is
given for $t\neq k\pi$ ($k$ integer) by the explicit formula%
\[
\widehat{S}_{t}\phi_{0}^{\hbar}(x)=i^{m(t)}\left(  \tfrac{1}{2\pi i|\sin
t|}\right)  ^{1/2}\int_{-\infty}^{\infty}e^{2\pi iW(x,x^{\prime},t)}\phi
_{0}^{\hbar}(x^{\prime})\mathrm{d}x^{\prime}%
\]
where $m(t)$ is an integer (the \textquotedblleft Maslov
index\textquotedblright) and%
\[
W(x,x^{\prime},t)=\frac{1}{2\sin t}\left(  (x^{2}+x^{\prime2})\cos
t-2xx^{\prime}\right)  .
\]

\begin{remark}
The metaplectic operators $\widehat{S}_{t}$ are the \textquotedblleft
fractional Fourier transforms\textquotedblright\ familiar from time-frequency
analysis (see e.g. Almeida \cite{almeida}, Namias \cite{namias}).
\end{remark}

Applying Proposition \ref{prop2} we recover without any calculation the
results of Kaiser \cite{kaiser} (Theorem 1 and Corollary 2) about rotations of
Gabor frames; in our notation:

\begin{corollary}
Let $\mathcal{G}(\phi,\Lambda)$ be a frame; then $\mathcal{G}(\widehat{S}%
\phi,S\Lambda)$ is a frame for every $\widehat{S}\in\operatorname*{Mp}(n)$.
\end{corollary}

\section{Symplectic Deformations of Gabor Frames\label{secsymp1}}

The symplectic covariance property of Gabor frames studied above can be
interpreted as a first result on Hamiltonian deformations of frames because,
as we will see, every symplectic matrix is the value of the flow (at some time
$t$) of a Hamiltonian function which is a homogeneous quadratic polynomial
(with time-depending coefficients) in the coordinates $x_{j},p_{k} $. We will
in fact extend this result to deformations by affine flows corresponding to
the case\ where the Hamiltonian is an arbitrary quadratic function of these coordinates.

\subsection{The linear case}

The first example in subsection \ref{subsecex} (the fractional Fourier
transform) can be interpreted as a statement about continuous deformations of
Gabor frames. For instance, assume that $S_{t}=e^{tX}$, $X$ in the Lie algebra
$\mathfrak{sp}(n)$ of the symplectic group $\operatorname*{Sp}(n)$ (it is the
algebra of all $2n\times2n$ matrices $X$ such that $XJ+JX^{T}=0$; when $n=1$
it reduces to the condition $\operatorname*{Tr}X=0$; see \cite{Folland,Birk}).
It is then easy to check that $S_{t}$ can be identified with the flow
determined by the Hamilton equations $\dot{z}=J\partial_{z}H$ for the
function
\begin{equation}
H(z)=-\tfrac{1}{2}z^{T}(JX)z \label{hamq2}%
\end{equation}
that is
\begin{equation}
\frac{\mathrm{d}}{\mathrm{d}t}S_{t}=XS_{t}. \label{xs1}%
\end{equation}
A fundamental observation is now that to the path of symplectic matrices
$t\longmapsto S_{t}$, $0\leq t\leq1$ corresponds a unique path $t\longmapsto
\widehat{S}_{t}$, $0\leq t\leq1$, of metaplectic operators such that
$\widehat{S}_{0}=I$ and $\widehat{S}_{1}=\widehat{S}$ (see Appendix A). This
path satisfies the operator Schr\"{o}dinger equation%
\begin{equation}
i\hbar\frac{\mathrm{d}}{\mathrm{d}t}\widehat{S}_{t}=\widehat{H}\widehat{S}_{t}
\label{schrop}%
\end{equation}
where $\widehat{H}$ \ is the Weyl quantization of the function $H$. Collecting
these facts, one sees that $\mathcal{G}(\widehat{S}\phi_{0}^{\hbar}%
,S\Lambda_{\alpha\beta})\mathcal{\ }$is obtained from the initial Gabor frame
$\mathcal{G}(\phi_{0}^{\hbar},\Lambda_{\alpha\beta}) $ by a smooth deformation%
\begin{equation}
t\longmapsto\mathcal{G}(\widehat{S}_{t}\phi_{0}^{\hbar},S_{t}\Lambda
_{\alpha\beta})\text{ \ , \ }0\leq t\leq1. \label{def1}%
\end{equation}
More generally, let $S$ be an arbitrary element of the symplectic group
$\operatorname*{Sp}(n)$. The latter is connected so there exists a $C^{1}$
path (in fact infinitely many) $t\longmapsto S_{t}$, $0\leq t\leq1$, joining
the identity to $S$ in $\operatorname*{Sp}(n)$. An essential result,
generalizing the observations above, is the following:

\begin{proposition}
\label{proph}Let $t\longmapsto S_{t}$, $0\leq t\leq1$, be a path in
$\operatorname*{Sp}(n)$ such that $S_{0}=I$ and $S_{1}=S$. There exists a
Hamiltonian function $H=H(z,t)$ such that $S_{t}$ is the phase flow determined
by the Hamilton equations $\dot{z}=J\partial_{z}H$. Writing
\begin{equation}
S_{t}=%
\begin{pmatrix}
A_{t} & B_{t}\\
C_{t} & D_{t}%
\end{pmatrix}
\label{st1}%
\end{equation}
the Hamiltonian function is the quadratic form%
\begin{multline}
H(z,t)=\tfrac{1}{2}(\dot{D}_{t}C_{t}^{T}-\dot{C}_{t}D_{t}^{T})x^{2}%
\label{hamzi}\\
+(\dot{D}_{t}A_{t}^{T}-\dot{C}_{t}B_{t}^{T})p\cdot x+\tfrac{1}{2}(\dot{B}%
_{t}A_{t}^{T}-\dot{A}_{t}B_{t}^{T})p^{2}%
\end{multline}
where $\dot{A}_{t}=\mathrm{d}A_{t}/\mathrm{d}t$, etc.
\end{proposition}

\begin{proof}
The matrices $S_{t}$ being symplectic we have $S_{t}^{T}JS_{t}=J$.
Differentiating both side of this equality with respect to $t$ and setting we
get $\dot{S}_{t}^{T}JS_{t}+S_{t}^{T}J\dot{S}_{t}=0$ or, equivalently,
\[
J\dot{S}_{t}S_{t}^{-1}=-(S_{t}^{T})^{-1}\dot{S}_{t}^{T}J=(J\dot{S}_{t}%
S_{t}^{-1})^{T}.
\]
This equality can be rewritten $J\dot{S}_{t}S_{t}^{-1}=(J\dot{S}_{t}S_{t}%
^{-1})^{T}$ hence the matrix $J\dot{S}_{t}S_{t}^{-1}$ is symmetric. Set
$J\dot{S}_{t}S_{t}^{-1}=M_{t}=M_{t}^{T}$; then%
\begin{equation}
\dot{S}_{t}=X_{t}S_{t}\text{ \ , \ }X_{t}=-JM_{t} \label{xs2}%
\end{equation}
(it reduces to formula (\ref{xs1}) when $M_{t}$ is time-independent). Define
now
\[
H(z,t)=-\tfrac{1}{2}z^{T}(JX_{t})z;
\]
using (\ref{xs2}) one verifies that the phase flow determined by $H$ consists
precisely of the symplectic matrices $S_{t}$ and that $H$ is given by formula
(\ref{hamzi}).
\end{proof}

\begin{remark}
Formula (\ref{hamzi}) also follows from the more general formula (\ref{hzt})
about Hamiltonian isotopies in Proposition (\ref{th2}) below.
\end{remark}

Exactly as above, to the path of symplectic matrices $t\longmapsto S_{t}$
corresponds a path $t\longmapsto\widehat{S}_{t}$ of metaplectic operators such
that $\widehat{S}_{0}=I$ and $\widehat{S}_{1}=\widehat{S}$ satisfying the
Schr\"{o}dinger equation (\ref{schrop}). Thus, it makes sense to consider
smooth deformations (\ref{def1}) for arbitrary symplectic paths. This
situation will be generalized to the nonlinear case in a moment.

\subsection{Translations of Gabor systems\label{sectrans}}

A particular simple example of transformation is that of the translations
$T(z_{0}):z\longmapsto z+z_{0}$ in $\mathbb{R}^{2n}$. On the operator level
they correspond to the Heisenberg--Weyl operators $\widehat{T}^{\hbar}(z_{0}%
)$. This correspondence is very easy to understand in terms of
\textquotedblleft quantization\textquotedblright: for fixed\ $z_{0}$ consider
the Hamiltonian function
\[
H(z)=\sigma(z,z_{0})=p\cdot x_{0}-p_{0}\cdot x.
\]
The corresponding Hamilton equations are just $\dot{x}=x_{0}$, $\dot{p}=p_{0}
$ whose solutions are $x(t)=x(0)+tx_{0}$ and $p(t)=p(0)+tp_{0}$, that is
$z(t)=T(tz_{0})z(0)$. Let now%
\[
\widehat{H}=\sigma(\widehat{z},z_{0})=(-i\hbar\partial_{x})\cdot x_{0}%
-p_{0}\cdot x.
\]
be the \textquotedblleft quantization\textquotedblright\ of $H$, and consider
the Schr\"{o}dinger equation%
\[
i\hslash\partial_{t}\phi=\sigma(\widehat{z},z_{0})\phi.
\]
Its solution is given by
\[
\phi(x,t)=e^{-t\sigma(\widehat{z},z_{0})/\hbar}\phi(x,0)=\widehat{T}^{\hbar
}(tz_{0})\phi(x,0)
\]
(the second equality can be verified by a direct calculation, or using the
Campbell--Hausdorff formula \cite{Folland,Birk,Birkbis,Littlejohn}).

Translations act in a particularly simple way on Gabor frames; we write
$T(z_{1})\Lambda=\Lambda+z_{1}$.

\begin{proposition}
\label{propzozun}Let $z_{0},z_{1}\in\mathbb{R}^{2n}$. A Gabor system
$\mathcal{G}(\phi,\Lambda)$ is a $\hbar$-frame if and only if $\mathcal{G}%
(\widehat{T}^{\hbar}(z_{0})\phi,T(z_{1})\Lambda)$ is a $\hbar$-frame; the
frame bounds are in this case the same for all values of $z_{0},z_{1}$.
\end{proposition}

\begin{proof}
We will need the following well-known \cite{Folland,Birk,Birkbis,Littlejohn}
properties of the Heisenberg--Weyl operators:%
\begin{align}
\widehat{T}^{\hbar}(z)\widehat{T}^{\hbar}(z^{\prime})  &  =e^{i\sigma
(z,z^{\prime})/\hbar}\widehat{T}^{\hbar}(z^{\prime})\widehat{T}^{\hbar
}(z)\label{hwt1}\\
\widehat{T}^{\hbar}(z+z^{\prime})  &  =e^{-i\sigma(z,z^{\prime})/2\hbar
}\widehat{T}^{\hbar}(z)\widehat{T}^{\hbar}(z^{\prime}). \label{hwt2}%
\end{align}
Assume first $z_{1}=0$ and let us prove that $\mathcal{G}(\widehat{T}^{\hbar
}(z_{0})\phi,\Lambda)$ is a $\hbar$-frame if and only $\mathcal{G}%
(\phi,\Lambda)$ is. We have, using formula (\ref{hwt1}) and the unitarity of
$\widehat{T}^{\hbar}(z_{0})$,%
\begin{align*}
\sum_{z\in\Lambda}|(\psi|\widehat{T}^{\hbar}(z)\widehat{T}^{\hbar}(z_{0}%
)\phi)|^{2}  &  =\sum_{z\in\Lambda}|(\psi|e^{i\sigma(z,z_{0})/\hbar}%
\widehat{T}^{\hbar}(z_{0})\widehat{T}^{\hbar}(z)\phi)|\\
&  =\sum_{z\in\Lambda}|(\psi|\widehat{T}^{\hbar}(z_{0})\widehat{T}^{\hbar
}(z)\phi)|\\
&  =\sum_{z\in\Lambda}|(\widehat{T}^{\hbar}(-z_{0})\psi|\widehat{T}^{\hbar
}(z)\phi)|;
\end{align*}
it follows that the inequality%
\[
a||\psi||^{2}\leq\sum_{z\in\Lambda}|(\psi|\widehat{T}^{\hbar}(z)\widehat
{T}^{\hbar}(z_{0})\phi)|^{2}\leq b||\psi||^{2}%
\]
is equivalent to%
\[
a||\psi||^{2}\leq\sum_{z\in\Lambda}|(\psi|\widehat{T}^{\hbar}(z)\phi)|^{2}\leq
b||\psi||^{2}%
\]
hence our claim in the case $z_{1}=0$. We next assume that $z_{0}=0$; we have,
using this time formula (\ref{hwt2}),
\begin{align*}
\sum_{z\in T(z_{1})\Lambda}|(\psi|\widehat{T}^{\hbar}(z)\phi)|^{2}  &
=\sum_{z\in\Lambda}|(\psi|\widehat{T}^{\hbar}(z+z_{1})\phi)|^{2}\\
&  =\sum_{z\in\Lambda}|(\psi|\widehat{T}^{\hbar}(z_{1})\widehat{T}^{\hbar
}(z)\phi)|^{2}\\
&  =\sum_{z\in\Lambda}|(\widehat{T}^{\hbar}(-z_{1})\psi|\widehat{T}^{\hbar
}(z)\phi)|^{2}%
\end{align*}
and one concludes as in the case $z_{1}=0$. The case of arbitrary $z_{0}%
,z_{1}$ immediately follows.
\end{proof}

Identifying the group of translations with $\mathbb{R}^{2n}$ the inhomogeneous
(or affine) symplectic group $\operatorname*{ISp}(n)$ \cite{burdet,Folland} is
the semi-direct product $\operatorname*{Sp}(n)\ltimes\mathbb{R}^{2n}$; the
group law is given by%
\[
(S,z)(S^{\prime},z^{\prime})=(SS^{\prime},z+Sz^{\prime}).
\]

Using the conjugation relation%
\begin{equation}
S^{-1}T(z_{0})S=T(S^{-1}z_{0}) \label{sts1}%
\end{equation}
one checks that $\operatorname*{ISp}(n)$ is isomorphic to the group of all
affine transformations of $\mathbb{R}^{2n}$ of the type $ST(z_{0})$ (or
$T(z_{0})S$) where $S\in\operatorname*{Sp}(n)$.

The group $\operatorname*{ISp}(n)$ appears in a natural way when one considers
Hamiltonians of the type%
\begin{equation}
H(z,t)=\tfrac{1}{2}M(t)z\cdot z+m(t)\cdot z \label{hmm}%
\end{equation}
where $M(t)$ is symmetric and $m(t)$ is a vector. In fact, the phase flow
determined by the Hamilton equations for (\ref{hmm}) consists of elements of
$\operatorname*{ISp}(n)$. Assume for instance that the coefficients $M$ and
$m$ are time-independent; the solution of Hamilton's equations $\dot
{z}=JMz+Jm$ is%
\begin{equation}
z_{t}=e^{tJM}z_{0}+(JM)^{-1}(e^{tJM}-I)JM \label{zt1}%
\end{equation}
provided that $\det M\neq0$. When $\det M=0$ the solution (\ref{zt1}) is still
formally valid and depends on the nilpotency degree of $X=JM$. Since
$X=JM\in\mathfrak{sp}(n)$ we have $S_{t}=e^{tX}\in\operatorname*{Sp}(n)$;
setting $\xi_{t}=X^{-1}(e^{tX}-I)u$ the flow $(f_{t}^{H})$ is thus given by%
\[
f_{t}^{H}=T(\xi_{t})S_{t}\in\operatorname*{ISp}(n).
\]

The metaplectic group $\operatorname*{Mp}(n)$ is a unitary representation of
the double cover $\operatorname*{Sp}_{2}(n)$ of $\operatorname*{Sp}(n)$ (see
Appendix A). There is an analogue when $\operatorname*{Sp}(n)$ is replaced
with $\operatorname*{ISp}(n)$: it is the Weyl-metaplectic group
$\operatorname*{WMp}(n)$, which consists of all products $\widehat{T}%
(z_{0})\widehat{S}$; notice that formula (\ref{symco}), which we can rewrite%
\begin{equation}
\widehat{S}^{-1}\widehat{T}^{\hbar}(z)\widehat{S}=\widehat{T}^{\hbar}%
(S^{-1}z).
\end{equation}
is the operator version of formula (\ref{sts1}) above.

\section{The Group $\operatorname*{Ham}(n)$\label{sechamn}}

In this section we review the basics of the modern theory of Hamiltonian
mechanics from the symplectic point of view; for details we refer to
\cite{ekho90,HZ,Polter}; we have also given elementary accounts in
\cite{Birk,Birkbis}.

\subsection{Hamiltonian flows: properties}

A Hamiltonian system (\ref{ham}) can be written in compact form as%
\begin{equation}
\dot{z}=J\partial_{z}H(z,t) \label{ham1}%
\end{equation}
where $J$ is the standard symplectic matrix. The Hamiltonian function $H$ is
assumed to be twice continuously differentiable in $z$, and continuous in $t$.
We denote by $f_{t}^{H}$ the mapping $\mathbb{R}^{2n}\longrightarrow
\mathbb{R}^{2n}$ which to an initial condition $z_{0}$ associates the value
$z=f_{t}^{H}(z_{0})$ of the solution to (\ref{ham1}) at time $t$. The family
$(f_{t}^{H})_{t}$ of all these mappings is called the phase flow determined by
the Hamiltonian system (\ref{ham1}).

It is often useful to replace the notion of flow as defined as above by that
of \emph{time-dependent flow} $(f_{t,t^{\prime}}^{H})$: $f_{t,t^{\prime}}^{H}$
is the function such that $f_{t,t^{\prime}}^{H}(z^{\prime})$ is the solution
of Hamilton's equations with $z(t^{\prime})=z^{\prime}$. Obviously
\begin{equation}
f_{t,t^{\prime}}^{H}=f_{t,0}^{H}\left(  f_{t^{\prime},0}^{H}\right)
^{-1}=f_{t}^{H}\left(  f_{t^{\prime}}^{H}\right)  ^{-1} \label{ftt'}%
\end{equation}
\ and the $f_{t,t^{\prime}}^{H}$ satisfy the groupoid property
\begin{equation}
f_{t,t^{\prime}}^{H}f_{t^{\prime},t^{\prime}}^{H}=f_{t,t^{\prime\prime}}%
^{H}\text{ \ , \ }f_{t,t}^{H}=I_{\mathrm{d}} \label{flow2}%
\end{equation}
for all $t$, $t^{\prime}$ and $t^{\prime\prime}$. Notice that it follows in
particular that $(f_{t,t^{\prime}}^{H})^{-1}=f_{t^{\prime},t}^{H}$.

An essential property which links Hamiltonian dynamics to symplectic geometry
is that each mapping $f_{t}^{H}$ is a diffeomorphism such that%
\begin{equation}
\left[  Df_{t}^{H}(z)\right]  ^{T}JDf_{t}^{H}(z)=Df_{t}^{H}(z)J\left[
Df_{t}^{H}(z)\right]  ^{T}=J. \label{jac1}%
\end{equation}
Here $Df_{t}^{H}(z)$ is the Jacobian matrix of the diffeomorphism $f_{t}^{H}$
calculated at the point $z=(x,p)$:%
\[
Df_{t}^{H}(z)=\frac{\partial z_{t}}{\partial z}=\frac{\partial(x_{t},p_{t}%
)}{\partial(x,p)}%
\]
if $z_{t}=f_{t}^{H}(z)$. The equality (\ref{jac1}) means that the matrix
$Df_{t}^{H}(z)$ is symplectic: $Df_{t}^{H}(z)\in\operatorname*{Sp}(n)$ for
every $z $ and $t$. Any diffeomorphism $f$ of phase space $\mathbb{R}^{2n}$
satisfying the condition%
\begin{equation}
Df(z)^{T}JDf(z)=Df(z)JDf(z)^{T}=J \label{symplecto}%
\end{equation}
is called a \emph{symplectomorphism}. Formula (\ref{jac1}) thus says that
Hamiltonian flows consist of symplectomorphisms, which is a well-known
property from classical mechanics \cite{Arnold,HZ,Polter}.

A remarkable fact is that composition and inversion of Hamiltonian flows also
yield Hamiltonian flows:

\begin{proposition}
\label{th1}Let $(f_{t}^{H})$ and $(f_{t}^{K})$ be the phase flows determined
by two Hamiltonian functions $H=H(z,t)$ and $K=K(z,t)$. We have
\begin{align}
f_{t}^{H}f_{t}^{K}  &  =f_{t}^{H\#K}\text{ \ \ with \ \ }%
H\#K(z,t)=H(z,t)+K((f_{t}^{H})^{-1}(z),t).\label{ch1}\\
(f_{t}^{H})^{-1}  &  =f_{t}^{\bar{H}}\text{ \ \ with \ \ }\bar{H}%
(z,t)=-H(f_{t}^{H}(z),t). \label{ch2}%
\end{align}

\end{proposition}

\begin{proof}
It is based on the transformation properties of the Hamiltonian fields
$X_{H}=J\partial_{z}H$ under diffeomorphisms; see \cite{Birk,HZ,Polter} for
detailed proofs.
\end{proof}

We notice that even if $H$ and $K$ are time-independent Hamiltonians, then
$H\#K$ and $\bar{H}$ are time-dependent.

\subsection{Hamiltonian isotopies}

Formula (\ref{symplecto}) above shows, using the chain rule, that the
symplectomorphisms of $\mathbb{R}^{2n}$ form a group, which we will denote by
$\operatorname*{Symp}(n)$.

Let us now focus on the Hamiltonian case. We will call a symplectomorphism $f
$ such that $f=f_{t}^{H}$ for some Hamiltonian function $H$ and time $t=1$ a
\emph{Hamiltonian symplectomorphism}. The choice of time $t=1$ in this
definition is of course arbitrary, and can be replaced with any other choice
$t=a$: we have $f=f_{t_{a}}^{H_{a}}$ where $H_{a}(z,t)=aH(z,at).$

Hamiltonian symplectomorphisms form a subgroup $\operatorname*{Ham}(n)$ of the
group $\operatorname*{Symp}(n)$ of all symplectomorphisms; it is in fact a
normal subgroup of $\operatorname*{Symp}(n)$ as follows from the formula%
\[
g^{-1}f_{t}^{H}g=f_{t}^{H\circ g}%
\]
valid for every symplectomorphism $g$ of $\mathbb{R}^{2n}$
\cite{HZ,Birk,Birkbis}. That $\operatorname*{Ham}(n)$ really is a group
follows from the two formulas (\ref{ch1}) and (\ref{ch2}) in Proposition
\ref{th1} above.

The following result is, in spite of its simplicity, a deep statement about
the structure of the group $\operatorname*{Ham}(n)$. It says that every
continuous path of Hamiltonian transformations passing through the identity is
itself the phase flow determined by a certain Hamiltonian function.

\begin{proposition}
\label{th2}Let $(f_{t})_{t}$ be a smooth one-parameter family of Hamiltonian
transformations such that $f_{0}=I_{\mathrm{d}}$. There exists a Hamiltonian
function $H=H(z,t)$ such that $f_{t}=f_{t}^{H}$. More precisely, $(f_{t})_{t}$
is the phase flow determined by the Hamiltonian function%
\begin{equation}
H(z,t)=-\int_{0}^{1}z^{T}J\left(  \dot{f}_{t}\circ f_{t}^{-1}\right)  (\lambda
z)\mathrm{d}\lambda\label{hzt}%
\end{equation}
where $\dot{f}_{t}=$\textrm{$d$}$f_{t}/\mathrm{d}t$.
\end{proposition}

We refer to Wang \cite{Wang} who gives an elementary proof of formula
(\ref{hzt}). The result goes back to the seminal paper of Banyaga
\cite{Banyaga}, but the idea is already present in Arnold \cite{Arnold} (p.
269) who uses the apparatus of generating functions.

We will a call smooth path $(f_{t})_{t}$ in $\operatorname*{Ham}(n)$ joining
the identity to some element $f\in\operatorname*{Ham}(n)$ a \emph{Hamiltonian
isotopy.}

\subsection{Symplectic algorithms\label{secsympalg}}

Symplectic integrators are designed for the numerical solution of Hamilton's
equations; they are algorithms which preserve the symplectic character of
Hamiltonian flows. The literature on the topic is immense; a well-cited paper
is Channel and Scovel \cite{channel}. Among many recent contributions, a
highlight is the recent treatise \cite{kang} by Kang Feng, Mengzhao Qin; also
see the comprehensive paper by Xue-Shen Liu \textit{et al.} \cite{xue}, and
Marsden's online lecture notes \cite{Marsden} (Chapter 9).

Let $(f_{t}^{H})$ be a Hamiltonian flow; we assume first that $H$ is
time-independent so that we have the one-parameter group property $f_{t}%
^{H}f_{t^{\prime}}^{H}=f_{t+t^{\prime}}^{H}$. Choose an initial value $z_{0}$
at time $t=0$. A mapping $f_{\Delta t}$ on $\mathbb{R}^{2n}$ is an algorithm
with time step-size $\Delta t$ for $(f_{t}^{H})$ if we have%
\[
f_{\Delta t}^{H}(z)=f_{\Delta t}(z)+O(\Delta t^{k});
\]
the number $k$ (usually an integer $\geq1$) is called the order of the
algorithm. In the theory of Hamiltonian systems one requires that $f_{\Delta
t}$ be a symplectomorphism; $f_{\Delta t}$ is then called a symplectic
integrator. One of the basic properties one is interested in is convergence:
setting $\Delta t=t/N$ ($N$ an integer) when do we have $\lim_{N\rightarrow
\infty}(f_{t/N})^{N}(z)=f_{t}^{H}(z)$? One important requirement is stability,
\textit{i.e.} $(f_{t/N})^{N}(z)$ must remain close to $z$ for small $t$ (see
Chorin \textit{et al}. \cite{chorin}).

Here are two elementary examples of symplectic integrators. We assume that the
Hamiltonian $H$ has the physical form%
\[
H(x,p)=U(p)+V(x).
\]

\begin{itemize}
\item \textit{First order algorithm}. One defines $(x_{k+1},p_{k+1})=f_{\Delta
t}(x_{k},p_{k})$ by
\begin{align*}
x_{k+1}  &  =x_{k}+\partial_{p}U(p_{k}-\partial_{x}V(x_{k})\Delta t)\Delta t\\
p_{k+1}  &  =p_{k}-\partial_{x}V(x_{k})\Delta t.
\end{align*}

\item \textit{Second order algorithm. }Setting%
\[
x_{k}^{\prime}=x_{k}+\tfrac{1}{2}\partial_{p}U(p_{k})
\]
we take%
\begin{align*}
x_{k+1}  &  =x_{k}^{\prime}+\tfrac{1}{2}\partial_{p}U(p_{k})\\
p_{k+1}  &  =p_{k}-\partial_{x}V(x_{k}^{\prime})\Delta t.
\end{align*}

\end{itemize}

One can show, using Proposition \ref{th2} (Wang \cite{Wang}), that both
schemes are not only symplectic, but also Hamiltonian. For instance, for the
first order algorithm above, we have $f_{\Delta t}=f_{\Delta t}^{K}$ where $K
$ is the now time-dependent Hamiltonian%
\begin{equation}
K(x,p,t)=U(p)+V(x-\partial_{p}U(p)t). \label{KUV}%
\end{equation}

When the Hamiltonian $H$ is itself time-dependent its flow does no longer
enjoy the group property $f_{t}^{H}f_{t^{\prime}}^{H}=f_{t+t^{\prime}}^{H}$,
so one has to redefine the notion of algorithm in some way. This can be done
by considering the time-dependent flow $(f_{t,t^{\prime}}^{H})$ defined by
(\ref{ftt'}): $f_{t,t^{\prime}}^{H}=f_{t}^{H}\left(  f_{t^{\prime}}%
^{H}\right)  ^{-1}$. One then uses the following trick: define the
\emph{suspended flow }$(\widetilde{f_{t}^{H}})$ by the formula
\begin{equation}
\widetilde{f_{t}^{H}}(z^{\prime},t^{\prime})=(f_{t+t^{\prime},t^{\prime}}%
^{H}(z^{\prime}),t+t^{\prime}); \label{susp}%
\end{equation}
one verifies that the mappings $\widetilde{f_{t}^{H}}:$ $\mathbb{R}^{2n}%
\times\mathbb{R\longrightarrow R}^{2n}\times\mathbb{R}$ (the \textquotedblleft
extended phase space\textquotedblright) satisfy the one-parameter group law
$\widetilde{f_{t}^{H}}\widetilde{f_{t^{\prime}}^{H}}=\widetilde{f_{t+t^{\prime
}}^{H}}$ and one may then define a notion of algorithm approximating
$\widetilde{f_{t}^{H}}$ (see Struckmeier \cite{struck} for the extended phase
space approach).

\section{Hamiltonian Deformations of Gabor Systems\label{sechade}}

Let $f\in\operatorname*{Ham}(n)$ and $(f_{t})_{0\leq t\leq1}$ a Hamiltonian
isotopy joining the identity to $f$; in view of Proposition \ref{th2} there
exists a Hamiltonian function $H$ such that $f_{t}=f_{t}^{H}$ for $0\leq
t\leq1$. We want to study the deformation of a $\hbar$-Gabor frame
$\mathcal{G}(\phi,\Lambda)$ by $(f_{t})_{0\leq t\leq1}$; that is we want to
define a deformation%
\begin{equation}
\mathcal{G}(\phi,\Lambda)\overset{f_{t}}{\longrightarrow}\mathcal{G}%
(\widehat{U}_{t}\phi,f_{t}\Lambda); \label{A}%
\end{equation}
here $\widehat{U}_{t}$ is an (unknown) operator associated in some (yet
unknown) way with $f_{t}$. The fact that when $f_{t}=S_{t}\in
\operatorname*{Sp}(n) $ we have%
\begin{equation}
\mathcal{G}(\phi,\Lambda)\overset{S_{t}}{\longrightarrow}\mathcal{G}%
(\widehat{S}_{t}\phi,S_{t}\Lambda); \label{B}%
\end{equation}
where $\widehat{S}_{t}\in\operatorname*{Mp}(n)$ with $S_{t}=\pi^{\hslash
}(\widehat{S}_{t})$ suggests that:

\begin{itemize}
\item \emph{The operators }$\widehat{U}_{t}$\emph{\ should be unitary;}

\item \emph{The deformation (\ref{A}) should reduce to (\ref{B}) when the
isotopy }$(f_{t})_{0\leq t\leq1}$ lies in $\operatorname*{Sp}(n)$.
\end{itemize}

The following property of the metaplectic representation gives us a clue. Let
$(S_{t})$ be a Hamiltonian isotopy in $\operatorname*{Sp}(n)\subset
\operatorname*{Ham}(n)$. We have seen in Proposition \ref{proph} that there
exists a Hamiltonian function%
\[
H(z,t)=\tfrac{1}{2}M(t)z\cdot z
\]
with associated phase flow precisely $(S_{t})$. Consider now the
Schr\"{o}dinger equation%
\[
i\hbar\frac{\partial\psi}{\partial t}=\widehat{H}\psi\text{ \ , \ }\psi
(\cdot,0)=\psi_{0}%
\]
where $\widehat{H}$ is the Weyl quantization of $H$ (it is a formally
self-adjoint operator since $H$ is real). It is well-known
\cite{Birk,Birkbis,Folland} that $\psi=\widehat{S}_{t}\psi_{0}$ where
$(\widehat{S}_{t})_{t}$ is the unique path in $\operatorname*{Mp}(n)$ passing
through the identity and covering $(S_{t})$. This suggests that we should
choose $(\widehat{U}_{t})_{t}$ in the following way: let $H$ be the
Hamiltonian function determined by the Hamiltonian isotopy $(f_{t})$:
$f_{t}=f_{t}^{H}$. Then quantize $H$ into a operator $\widehat{H}$ using the
Weyl correspondence, and let $(\widehat{U}_{t})$ be the solution of
Schr\"{o}dinger's equation%
\[
i\hbar\frac{\mathrm{d}}{\mathrm{d}t}\widehat{U}_{t}=\widehat{H}\widehat{U}%
_{t}\text{ \ , \ }F_{0}=I_{\mathrm{d}}.
\]
The operators $\widehat{U}_{t}$ are unitary. Let in fact $u(t)=(\widehat
{U}_{t}\psi|\widehat{U}_{t}\psi)$ where $\psi$ is in the domain of
$\widehat{H}$ (we assume it contains $\mathcal{S}(\mathbb{R}^{n})$); we have%
\[
i\hbar\dot{u}(t)=(\widehat{H}\widehat{U}_{t}\psi|\widehat{U}_{t}%
\psi)-(\widehat{U}_{t}\psi|\widehat{H}\widehat{U}_{t}\psi)=0
\]
since $\widehat{H}$ is (formally) self-adjoint; it follows that $(\widehat
{U}_{t}\psi|\widehat{U}_{t}\psi)=(\psi|\psi)$.

While definition (\ref{B}) of a Hamiltonian deformation of a Gabor system is
\textquotedblleft reasonable\textquotedblright, its practical implementation
is difficult because it requires the solution of a Schr\"{o}dinger equation.
We will try to find a weaker, more tractable definition of the correspondence
\emph{(\ref{A}), }which is easier to implement numerically.

\subsection{The semiclassical approach}

The idea of this method comes from semiclassical mechanics; historically it
seems to be due to Heller \cite{heller}. Hagedorn \cite{hado1,hado2} has given
the method a firm mathematical basis; also see Littlejohn's review paper
\cite{Littlejohn} and the construction in \cite{Naza}.

For fixed $z_{0}$ we set $z_{t}=f_{t}^{H}(z_{0})$ and define the new Hamilton
function
\begin{equation}
H_{z_{0}}(z,t)=(\partial_{z}H)(z_{t},t)(z-z_{t})+\tfrac{1}{2}D_{z}^{2}%
H(z_{t},t)(z-z_{t})^{2}; \label{k}%
\end{equation}
it is the Taylor series of $H$ at $z_{t}$ with terms of order $0$ and $>2$
suppressed. The corresponding Hamilton equations are
\begin{equation}
\dot{z}=J\partial_{z}H(z_{t},t)+JD_{z}^{2}H(z_{t},t)(z-z_{t}). \label{hamk}%
\end{equation}
We make the following obvious but essential observation: in view of the
uniqueness theorem for the solutions of Hamilton's equations, the solution of
(\ref{hamk}) with initial value $z_{0}$ is the same as that of the Hamiltonian
system
\begin{equation}
\dot{z}(t)=J\partial_{z}H(z(t),t) \label{hamo}%
\end{equation}
with $z(0)=z_{0}$. Denoting by $(f_{t}^{H_{z_{0}}})$ the Hamiltonian flow
determined by $H_{z_{0}}$ we thus have $f_{t}^{H}(z_{0})=f_{t}^{H_{z_{0}}%
}(z_{0})$. More generally, the flows $(f_{t}^{H_{z_{0}}})$ and $(f_{t}^{H})$
are related by a simple formula:

\begin{proposition}
The solutions of Hamilton's equations (\ref{hamk}) and (\ref{hamo}) are
related by the formula%
\begin{equation}
z(t)=z_{t}+S_{t}(z(0)-z_{0}) \label{bf}%
\end{equation}
where $z_{t}=f_{t}^{H}(z_{0})$, $z(t)=f_{t}^{H}(z)$ and $(S_{t})_{t}$ is the
phase flow determined by the quadratic time-dependent Hamiltonian%
\begin{equation}
H^{0}(z,t)=\tfrac{1}{2}D_{z}^{2}H(z_{t},t)z\cdot z. \label{ho}%
\end{equation}
Equivalently,%
\begin{equation}
f_{t}^{H}(z)=T[z_{t}-S_{t}(z_{0})]S_{t}(z(0)) \label{bfbis}%
\end{equation}
where $T(\cdot)$ is the translation operator.
\end{proposition}

\begin{proof}
Let us set $u=z-z_{t}$. We have, taking (\ref{hamk}) into account,
\[
\dot{u}+\dot{z}_{t}=J\partial_{z}H(z(t),t)+JD_{z}^{2}H(z_{t},t)u
\]
that is, since $\dot{z}_{t}=J\partial_{z}H(z_{t},t)$,%
\[
\dot{u}=JD_{z}^{2}H(z_{t},t)u.
\]
It follows that $u(t)=S_{t}(u(0))$ and hence%
\[
z(t)=f_{t}^{H}(z_{0})+S_{t}u(0)=z_{t}-S_{t}(z_{0})+S_{t}(z(0))
\]
which is precisely (\ref{bf}).
\end{proof}

The nearby orbit method (at order $N=0$) consists in making the Ansatz that
the approximate solution to Schr\"{o}dinger's equation
\[
i\hslash\frac{\partial\psi}{\partial t}=\widehat{H}\psi\text{ \ ,\ \ }%
\psi(\cdot,0)=\phi_{z_{0}}^{\hbar}%
\]
where
\begin{equation}
\phi_{z_{0}}^{\hbar}=\widehat{T}^{\hbar}(z_{0})\phi_{0}^{\hbar} \label{b}%
\end{equation}
is the standard coherent state centered at $z_{0}$ is given by the formula \
\begin{equation}
\widetilde{\psi}(x,t)=e^{\frac{i}{\hbar}\gamma(t,z_{0})}\widehat{T}^{\hbar
}(z_{t})\widehat{S}_{t}(z_{0})\widehat{T}^{\hbar}(z_{0})^{-1}\phi_{z_{0}%
}^{\hbar} \label{nom}%
\end{equation}
where the phase $\gamma(t,z_{0})$ is the symmetrized action%
\begin{equation}
\gamma(t,z_{0})=\int_{0}^{t}\left(  \tfrac{1}{2}\sigma(z_{t^{\prime}},\dot
{z}_{t^{\prime}})-H(z_{t^{\prime}},t^{\prime})\right)  \mathrm{d}t^{\prime}
\label{phase}%
\end{equation}
calculated along the Hamiltonian trajectory leading from $z_{0}$ at time
$t_{0}=0$ to $z_{t}$ at time $t$. One shows that under suitable conditions on
the Hamiltonian $H$ the approximate solution satisfies, for $|t|\leq T$, an
estimate of the type%
\begin{equation}
||\psi(\cdot,t)-\widetilde{\psi}(\cdot,t)||\leq C(z_{0},T)\sqrt{\hbar}|t|
\label{psio}%
\end{equation}
where $C(z_{0},T)$ is a positive constant depending only on the initial point
$z_{0}$ and the time interval $[-T,T]$ (Hagedorn \cite{hado1,hado2}).

\begin{remark}
Formula (\ref{nom}) shows that the solution of Schr\"{o}dinger's equation with
initial datum $\phi_{0}^{\hbar}$ is approximately the Gaussian obtained by
propagating $\phi_{0}^{\hbar}$ along the Hamiltonian trajectory starting from
$z=0$ while deforming it using the metaplectic lift of the linearized flow
around this point.
\end{remark}

\subsection{Application to Gabor frames}

Let us state and prove the main results of this paper.

In what follows we consider a Gaussian Gabor system $\mathcal{G}(\phi
_{0}^{\hbar},\Lambda)$; applying the nearby orbit method to $\phi_{0}^{\hbar}$
yields the approximation%
\begin{equation}
\phi_{t}^{\hbar}=e^{\frac{i}{\hbar}\gamma(t,0)}\widehat{T}^{\hbar}%
(z_{t})\widehat{S}_{t}\phi_{0}^{\hbar} \label{fth}%
\end{equation}
where we have set $\widehat{S}_{t}=\widehat{S}_{t}(0)$. Let us consider the
Gabor system $\mathcal{G}(\phi_{t}^{\hbar},\Lambda_{t})$ where $\Lambda
_{t}=f_{t}^{H}(\Lambda)$.

\begin{proposition}
\label{propmain1}The Gabor system $\mathcal{G}(\phi_{t}^{\hbar},\Lambda_{t})$
is a Gabor $\hbar$-frame if and only if $\mathcal{G}(\phi_{0}^{\hbar}%
,\Lambda)$ is a a Gabor $\hbar$-frame; when this is the case both frames have
the same bounds.
\end{proposition}

\begin{proof}
Writing
\begin{equation}
I_{t}(\psi)=\sum_{z\in\Lambda_{t}}|(\psi|\widehat{T}^{\hbar}(z)\phi_{t}%
^{\hbar})|^{2}%
\end{equation}
we set out to show that the inequality
\begin{equation}
a||\psi||^{2}\leq I_{t}(\psi)\leq b||\psi||^{2} \label{ineq}%
\end{equation}
(for all $\psi\in L^{2}(\mathbb{R}^{n})$) holds for every $t$ if and only if
it holds for $t=0$ (for all $\psi\in L^{2}(\mathbb{R}^{n})$). In view of
definition (\ref{fth}) we have
\[
I_{t}(\psi)=\sum_{z\in\Lambda_{t}}|(\psi|\widehat{T}^{\hbar}(z)\widehat
{T}^{\hbar}(z_{t})\widehat{S}_{t}\phi_{0}^{\hbar})|^{2};
\]
the commutation formula (\ref{hwt1}) yields%
\[
\widehat{T}^{\hbar}(z)\widehat{T}^{\hbar}(z_{t})=e^{i\sigma(z,z_{t})/\hbar
}\widehat{T}^{\hbar}(z_{t})\widehat{T}^{\hbar}(z)
\]
and hence%
\begin{align*}
I_{t}(\psi)  &  =\sum_{z\in\Lambda_{t}}|(\psi|\widehat{T}^{\hbar}%
(z_{t}))\widehat{T}^{\hbar}(z)\widehat{S}_{t}\phi_{0}^{\hbar})|^{2}\\
&  =\sum_{z\in\Lambda}|(\psi|\widehat{T}^{\hbar}(z_{t})\widehat{T}^{\hbar
}(f_{t}^{H}(z))\widehat{S}_{t}\phi_{0}^{\hbar})|^{2}.
\end{align*}
Since $\widehat{T}^{\hbar}(z_{t})$ is unitary the inequality (\ref{ineq}) is
thus equivalent to%
\begin{equation}
a||\psi||^{2}\leq\sum_{z\in\Lambda}|(\psi|\widehat{T}^{\hbar}(f_{t}%
^{H}(z))\widehat{S}_{t}\phi_{0}^{\hbar})|^{2}\leq b||\psi||^{2}.
\label{ineqbis}%
\end{equation}
In view of formula (\ref{bf}) we have, since $S_{t}z_{0}=0$ because $z_{0}%
=0$,
\[
f_{t}^{H}(z)=S_{t}z+f_{t}^{H}(0)=S_{t}z+z_{t}%
\]
hence the inequality (\ref{ineqbis}) can be written%
\begin{equation}
a||\psi||^{2}\leq\sum_{z\in\Lambda}|(\psi|\widehat{T}^{\hbar}(S_{t}%
z+z_{t})\widehat{S}_{t}\phi_{0}^{\hbar})|^{2}\leq b||\psi||^{2}.
\label{ineqter}%
\end{equation}
In view of the product formula (\ref{hwt2}) for Heisenberg--Weyl operators we
have%
\[
\widehat{T}^{\hbar}(S_{t}z+z_{t})=e^{i\sigma(S_{t}z,z_{t})/2\hbar}\widehat
{T}^{\hbar}(z_{t})\widehat{T}^{\hbar}(S_{t}z)
\]
so that (\ref{ineqter}) becomes%
\begin{equation}
a||\psi||^{2}\leq\sum_{z\in\Lambda}|(\psi|\widehat{T}^{\hbar}(z_{t}%
)\widehat{T}^{\hbar}(S_{t}z)\widehat{S}_{t}\phi_{0}^{\hbar})|^{2}\leq
b||\psi||^{2}; \label{inequart}%
\end{equation}
the unitarity of $\widehat{T}^{\hbar}(z_{t})$ implies that (\ref{inequart}) is
equivalent to%
\begin{equation}
a||\psi||^{2}\leq\sum_{z\in\Lambda}|(\psi|\widehat{T}^{\hbar}(S_{t}%
z)\widehat{S}_{t}\phi_{0}^{\hbar})|^{2}\leq b||\psi||^{2}. \label{inecinq}%
\end{equation}
Using the symplectic covariance formula (\ref{symco}) we have%
\[
\widehat{T}^{\hbar}(S_{t}z)\widehat{S}_{t}=\widehat{S}_{t}\widehat{T}^{\hbar
}(z)
\]
so that the inequality (\ref{inecinq}) can be written%
\[
a||\psi||^{2}\leq\sum_{z\in\Lambda}|(\psi|\widehat{S}_{t}\widehat{T}^{\hbar
}(z)\phi_{0}^{\hbar})|^{2}\leq b||\psi||^{2};
\]
since $\widehat{S}_{t}$ is unitary, this is equivalent to%
\[
a||\psi||^{2}\leq\sum_{z\in\Lambda}|(\psi|\widehat{T}^{\hbar}(z)\phi
_{0}^{\hbar})|^{2}\leq b||\psi||^{2}.
\]
The Proposition follows.
\end{proof}

The fact that we assumed that the window is the centered coherent state
$\phi_{0}^{\hbar}$ is not essential. For instance, Proposition \ref{propzozun}
shows that the result remains valid if we replace $\phi_{0}^{\hbar}$ with a
coherent state having arbitrary center, for instance $\phi_{z_{0}}^{\hbar
}=\widehat{T}^{\hbar}(z_{0})\phi_{0}^{\hbar}$. More generally:

\begin{corollary}
\label{propmain2}Let $\mathcal{G}(\phi,\Lambda)$ be a Gabor system where the
window $\phi$ is the Gaussian%
\begin{equation}
\phi_{M}^{\hbar}(x)=\left(  \tfrac{\det\operatorname{Im}M}{(\pi\hbar)^{n}%
}\right)  ^{1/4}e^{\frac{i}{2\hbar}Mx\cdot x} \label{fimgauss}%
\end{equation}
where $M=M^{T}$, $\operatorname{Im}M>0$. Then $\mathcal{G}(\phi_{t}^{\hbar
},\Lambda_{t})$ is a Gabor $\hbar$-frame if and only if it is the case for
$\mathcal{G}(\phi,\Lambda)$.
\end{corollary}

\begin{proof}
It follows from the properties of the action of the metaplectic group on
Gaussians (see Appendix A) that there exists $\widehat{S}\in\operatorname*{Mp}%
(n)$ such that $\phi_{M}^{\hbar}=\widehat{S}\phi_{0}^{\hbar}$. Let
$S=\pi^{\hbar}(\widehat{S})$ be the projection on $\operatorname*{Sp}(n)$ of
$\widehat{S}$; the Gabor system $\mathcal{G}(\phi_{M}^{\hbar},\Lambda)$ is a
$\hbar$-frame if and only if $\mathcal{G}(\widehat{S}^{-1}\phi_{M}^{\hbar
},S^{-1}\Lambda)=\mathcal{G}(\phi_{0}^{\hbar},S^{-1}\Lambda)$ is a $\hbar
$-frame in view of Proposition \ref{prop2}. The result now follows from
Proposition \ref{propmain1}.
\end{proof}

\section{Discussion and Additional Remarks}

We have given one working definition of the notion of Hamiltonian deformation
of a Gabor frame; this definition uses ideas from semiclassical mechanics.
However, we have used no approximations. We could therefore call this
deformation scheme \textquotedblleft weak Hamiltonian
deformation\textquotedblright. An important remark is that in all our results,
one can assume that the window $\phi$ belongs to the Feichtinger algebra
$S_{0}(\mathbb{R}^{n})$ (reviewed in Appendix B). This is due to the fact that
we have transformed the Gabor frames under consideration only by the phase
space translations $\widehat{T}^{\hbar}(z)$ and by metaplectic operators; it
turns out that $S_{0}(\mathbb{R}^{n})$ is the smallest Banach algebra
invariant under these operations, and thus semiclassical propagation preserve
the Feichtinger algebra (see de Gosson \cite{miop}). A consequence is that the
weak Hamiltonian deformation scheme behaves well with respect to the
Feichtinger algebra. It is unknown whether this property is conserved under
passage to the general definition (\ref{B}), that is%
\begin{equation}
\mathcal{G}(\phi,\Lambda)\overset{f_{t}}{\longrightarrow}\mathcal{G}%
(\widehat{U}_{t}\phi,f_{t}\Lambda) \label{gut}%
\end{equation}
where $\widehat{U}_{t}$ is the solution of the Schr\"{o}dinger equation
associated with the Hamiltonian operator corresponding to the Hamiltonian
isotopy $(f_{t})_{0\leq t\leq1}$. This because one does not know at the time
of writing if the solution to Schr\"{o}dinger equations with initial data in
$S_{0}(\mathbb{R}^{n})$ also is in $S_{0}(\mathbb{R}^{n})$ for given time $t$.

Since our definition of weak deformations was motivated by semiclassical
considerations one could perhaps consider refinements of this method using the
asymptotic expansions of Hagedorn \cite{hado1,hado2} and his followers; this
could then lead to \textquotedblleft higher order\textquotedblright\ weak
deformations, depending on the number of terms that are retained. Still, there
remains the question of the general definition (\ref{gut}) where the exact
quantum propagator is used. It would indeed be more intellectually (and also
probably practically!) satisfying to study this definition in detail. As we
said, we preferred in this first approach to consider a weaker version because
it is relatively easy to implement numerically using symplectic integrators.
The general case (\ref{gut}) is challenging, but not probably out of reach.
From a theoretical point of view, it amounts to construct an extension of the
metaplectic representation in the non-linear case; that such a representation
indeed exists has been shown in our paper with Hiley \cite{gohi11} (a caveat:
one sometimes finds in the physical literature a claim following which such an
extension could not be constructed; a famous theorem of Groenewold and Van
Hove being invoked. This is merely a misunderstanding of this theorem, which
only claims that there is no way to extend the metaplectic representation so
that the Dirac correspondence between Poisson brackets and commutators is
preserved). There remains the problem of how one could prove that the
deformation scheme (\ref{gut}) preserves the frame property; a possible
approach could consisting in using a time-slicing (as one does for symplectic
integrators); this would possibly also lead to some insight on whether the
Feichtinger algebra is preserved by general quantum evolution.

\section*{APPENDIX\ A: METAPLECTIC\ GROUP AND\ GAUSSIANS}

Let $\operatorname*{Mp}(n)$ be the metaplectic representation of the
symplectic group $\operatorname*{Sp}(n)$ (see \cite{Folland,Birk,Birkbis}); it
is a unitary representation of the double cover $\operatorname*{Sp}_{2}(n)$ of
$\operatorname*{Sp}(n)$: we have a short exact sequence%
\[
0\longrightarrow\mathbb{Z}_{2}\longrightarrow\operatorname*{Mp}(n)\overset
{\pi^{\hbar}}{\longrightarrow}\operatorname*{Sp}(n)\longrightarrow0
\]
where $\pi^{\hbar}:$ $\widehat{S}\longmapsto S$ is the covering projection; we
explain the appearance of the subscript $\hslash$ below. The metaplectic group
is generated by the following elementary operators:

\begin{itemize}
\item \textit{Fourier transform}:
\begin{equation}
\widehat{J}\psi=\left(  \tfrac{1}{2\pi i\hbar}\right)  ^{n/2}\int
_{{\mathbb{R}}^{n}}e^{-ixx^{\prime}/\hbar}\psi(x^{\prime})\mathrm{d}x^{\prime}
\label{mp1}%
\end{equation}
(notice the presence of the imaginary unit $i$ in the prefactor);

\item \textit{Unitary dilations}:
\begin{equation}
\widehat{M}_{L,m}\psi=i^{m}\sqrt{|\det L|}\psi(Lx)\text{ \ }(\det L\neq0)
\label{mp2}%
\end{equation}
where $m$ is an integer depending on the sign of $\det L$: $m\in\{0,2\}$ if
$\det L>0$ and $m\in\{1,3\}$ if $\det L<0$;

\item \textquotedblleft\textit{Chirps\textquotedblright}:
\begin{equation}
\widehat{V}_{P}=e^{-iPx^{2}/2\hbar}\psi(x)\text{ \ \ }(P=P^{T}). \label{mp3}%
\end{equation}
The projections on $\operatorname*{Sp}(n)$ of these operators are given by
$\pi^{\hbar}(\widehat{J})=J$, $\pi^{\hbar}(\widehat{M}_{L,m})=M_{L}$, and
$\pi^{\hbar}(\widehat{V}_{P})=V_{P}$ with
\begin{equation}
M_{L}=%
\begin{pmatrix}
L^{-1} & 0\\
0 & L^{T}%
\end{pmatrix}
\text{ \ , \ }V_{P}=%
\begin{pmatrix}
I & 0\\
-P & I
\end{pmatrix}
\label{proj}%
\end{equation}
(the matrices $V_{P}$ are sometimes called \textquotedblleft symplectic
shears\textquotedblright).
\end{itemize}

The projection of a covering group onto its base group is defined only up to
conjugation; our choice --and notation-- is here dictated by the fact that to
the $\hbar$-dependent operators (\ref{mp1}) and (\ref{mp2}) should correspond
the symplectic matrices (\ref{proj}). For instance, in time-frequency analysis
it is customary to make the choice $\hbar=1/2\pi$. Following formula relates
the projections $\pi^{\hbar}$ and $\pi=\pi^{1/2\pi}$: \
\begin{equation}
\pi(\widehat{S})=\pi^{\hbar}(\widehat{M}_{1/\sqrt{2\pi\hbar}}\widehat
{S}\widehat{M}_{\sqrt{2\pi\hbar}}) \label{pich}%
\end{equation}
where $\widehat{M}_{\sqrt{2\pi\hbar}}=\widehat{M}_{\sqrt{2\pi\hbar}I,0}$.

Metaplectic operators are not only unitary operators on $L^{2}(\mathbb{R}%
^{n})$ but also linear automorphisms of $\mathcal{S}(\mathbb{R}^{n})$ which
extend by duality to automorphisms of $\mathcal{S}^{\prime}(\mathbb{R}^{n})$.

There is an alternative way to describe the metaplectic group
$\operatorname*{Mp}(n)$. Let
\begin{equation}
W(x,x^{\prime})=\tfrac{1}{2}Px^{2}-Lx\cdot x^{\prime}+\tfrac{1}{2}Qx^{2}
\label{W}%
\end{equation}
where $P,L,Q$ are real $n\times n$ matrices, $P$ and $Q$ symmetric and $L$
invertible (we are writing $Px^{2}$ for $Px\cdot x$, etc.). Let $m$ be a
choice of $\arg\det L$ as in formula (\ref{mp2}); each $\widehat{S}%
\in\operatorname*{Mp}(n)$ is the product to two operators of the type%
\begin{equation}
\widehat{S}_{W,m}\psi(x)=\left(  \tfrac{1}{2\pi i\hbar}\right)  ^{n/2}%
i^{m}\sqrt{|\det L|}\int_{{\mathbb{R}}^{n}}e^{-iW(x,x^{\prime})/\hbar}%
\psi(x^{\prime})\mathrm{d}x^{\prime} \tag{A1}\label{swm}%
\end{equation}
(see \cite{Leray,Birk,Birkbis}). The operators $\widehat{S}_{W,m}$ can be
factorized as
\begin{equation}
\widehat{S}_{W,m}=\widehat{V}_{-P}\widehat{M}_{L,m}\widehat{J}\psi\widehat
{V}_{-Q} \tag{A2}\label{vpfmlvq}%
\end{equation}
and hence belong to $\operatorname*{Mp}(n)$. The projection $S_{W}=\pi^{\hbar
}(\widehat{S}_{W,m})$ is characterized by the condition%
\[
(x,p)=S_{W}(x^{\prime},p^{\prime})\Longleftrightarrow\left\{
\begin{array}
[c]{c}%
p=\partial_{x}W(x,x^{\prime})\\
p^{\prime}=-\partial_{x}W(x,x^{\prime})
\end{array}
\right.  ;
\]
this condition identifies $W$ with the generating function of first type,
familiar from Hamiltonian mechanics \cite{Arnold,Birk,Birkbis,Leray}. A
straightforward calculation using the expression (\ref{W}) of $W$ yields the
symplectic matrix%
\[
S_{W}=%
\begin{pmatrix}
L^{-1}Q & L^{-1}\\
PL^{-1}Q-L^{T} & PL^{-1}%
\end{pmatrix}
.
\]

The metaplectic group acts on Gaussian functions in a particularly simple way.
Let $M$ be a complex $n\times n$ matrix; we assume in fact that $M$ belongs to
the Siegel half-space
\[
\Sigma_{n}^{+}=\{M:M=M^{T},\operatorname{Im}M>0\}.
\]
We call generalized centered coherent state a Gaussian function of the type
\begin{equation}
\phi_{M}^{\hbar}(x)=\left(  \tfrac{\det\operatorname{Im}M}{(\pi\hbar)^{n}%
}\right)  ^{1/4}e^{\frac{i}{2\hbar}Mx\cdot x} \label{fmz}%
\end{equation}
and for $z_{0}\in\mathbb{R}^{2n}$ we set
\begin{equation}
\phi_{M,z_{0}}^{\hbar}=\widehat{T}^{\hbar}(z_{0})\phi_{M}^{\hbar} \label{fmzo}%
\end{equation}
(it is a Gaussian centered at the point $z_{0}$). The symplectic group
$\operatorname*{Sp}(n)$ acts transitively on the Siegel half-space via the law
\cite{Folland}
\begin{equation}
(S,M)\longmapsto\alpha(S)M=(C+DM)(A+BM)^{-1}. \label{mpc}%
\end{equation}
One can show \cite{Birkbis} that if $M=X+iY$ then
\begin{align}
X  &  =-(CA^{T}+DB^{T})(AA^{T}+BB^{T})^{-1}\label{xaa}\\
Y  &  =(AA^{T}+BB^{T})^{-1}. \label{xbb}%
\end{align}
This action induces in turn a transitive action
\[
(\widehat{S},\phi_{M,z_{0}}^{\hbar})\longmapsto\phi_{\alpha(S)M,Sz_{0}}%
^{\hbar}%
\]
of the metaplectic group $\operatorname*{Mp}(n)$ on the set $G_{n}$ of
Gaussians of the type (\ref{fmzo}). These actions make the following diagram
\[%
\begin{array}
[c]{ccc}%
\operatorname*{Mp}(n)\times G_{n} & \longrightarrow & G_{n}\\
\downarrow &  & \downarrow\\
\operatorname*{Sp}(n)\times\Sigma_{n}^{+} & \longrightarrow & \Sigma_{n}^{+}%
\end{array}
\]
commutative (the vertical arrows being the mappings $(\widehat{S}%
,\phi_{M,z_{0}}^{\hbar})\longmapsto(S,M)$ and $\phi_{M,z_{0}}^{\hbar
}\longmapsto M$, respectively).

The formulas above can be proven by using either the properties of the Wigner
transform, or by a calculation of Gaussian integrals using the operators
$\widehat{S}_{W,m}$ defined by formula (\ref{swm}).

\section*{APPENDIX\ B: FEICHTINGER'S\ ALGEBRA}

The Feichtinger algebra $S_{0}(\mathbb{R}^{n})$ was introduced in
\cite{hanskernel,Hans1,fei81}; it is an important particular case of the
modulation spaces defined by the same author; we refer to Gr\"{o}chenig's
treatise \cite{Gro} for a complete study of these important functional spaces.
Also see Feichtinger and Luef \cite{felu12} for an up to date concise review.

The Feichtinger algebra is usually defined in terms of short-time Fourier
transform%
\begin{equation}
V_{\phi}\psi(z)=\int_{\mathbb{R}^{n}}e^{-2\pi ip\cdot x^{\prime}}%
\psi(x^{\prime})\overline{\phi(x^{\prime}-x)}\mathrm{d}x^{\prime};
\label{stft}%
\end{equation}
which is related to the cross-Wigner transform by the formula%
\begin{equation}
W(\psi,\phi)(z)=\left(  \tfrac{2}{\pi\hbar}\right)  ^{n/2}e^{\frac{2i}{\hbar
}p\cdot x}V_{\phi_{\sqrt{2\pi\hbar}}^{\vee}}\psi_{\sqrt{2\pi\hbar}}%
(z\sqrt{\tfrac{2}{\pi\hbar}}) \label{wstft}%
\end{equation}
where $\psi_{\sqrt{2\pi\hbar}}(x)=\psi(x\sqrt{2\pi\hbar})$ and $\phi^{\vee
}(x)=\phi(-x)$; equivalently%
\begin{equation}
V_{\phi}\psi(z)=\left(  \tfrac{2}{\pi\hbar}\right)  ^{-n/2}e^{-i\pi p\cdot
x}W(\psi_{1/\sqrt{2\pi\hbar}},\phi_{1/\sqrt{2\pi\hbar}}^{\vee})(z\sqrt
{\tfrac{\pi\hbar}{2}}). \label{wstftbis}%
\end{equation}

The Feichtinger algebra $S_{0}(\mathbb{R}^{n})$ consists of all $\psi
\in\mathcal{S}^{\prime}(\mathbb{R}^{n})$ such that $V_{\phi}\psi\in
L^{1}(\mathbb{R}^{2n})$ for every window $\phi$. \ In view of the relations
(\ref{wstft}), (\ref{wstftbis}) this condition is equivalent to $W(\psi
,\phi)\in L^{1}(\mathbb{R}^{2n})$. A function $\psi\in L^{2}(\mathbb{R}^{n})$
belongs to $S_{0}(\mathbb{R}^{n})$ if and only if $W\psi\in L^{1}%
(\mathbb{R}^{2n})$; here $W\psi=W(\psi,\psi)$ is the usual Wigner function.
The number%
\begin{equation}
||\psi||_{\phi,S_{0}}^{\hbar}=||W(\psi,\phi)||_{L^{1}(\mathbb{R}^{2n})}%
=\int_{\mathbb{R}^{2n}}|W(\psi,\phi)(z)|\mathrm{d}z \label{wignorm}%
\end{equation}
is the norm of $\psi$ relative to the window $\phi$. We have the inclusions%
\begin{equation}
\mathcal{S}(\mathbb{R}^{n})\subset S_{0}(\mathbb{R}^{n})\subset C^{0}%
(\mathbb{R}^{n})\cap L^{1}(\mathbb{R}^{n})\cap F(L^{1}(\mathbb{R}^{n}))
\label{inclo}%
\end{equation}
where $F(L^{1}(\mathbb{R}^{n}))$ is the image of $L^{1}(\mathbb{R}^{n})$ by
the Fourier transform. One proves that $S_{0}(\mathbb{R}^{n})$ is an algebra,
both for pointwise multiplication and for convolution.

An essential property of the Feichtinger algebra is that it is closed under
the action of the Weyl-metaplectic group $\operatorname*{WMp}(n)$: if $\psi\in
S_{0}(\mathbb{R}^{n})$, $\widehat{S}\in\operatorname*{Mp}(n)$, and $z_{0}%
\in\mathbb{R}^{n}$ we have both $\widehat{S}\psi\in S_{0}(\mathbb{R}^{n})$ and
$\widehat{T}^{\hbar}(z_{0})\psi\in S_{0}(\mathbb{R}^{n})$. In particular
$\psi\in S_{0}(\mathbb{R}^{n})$ if and only if $F\psi\in S_{0}(\mathbb{R}%
^{n})$. One proves that $S_{0}(\mathbb{R}^{n})$ is the smallest Banach space
containing $\mathcal{S}(\mathbb{R}^{n})$ and having this property.

\begin{acknowledgement}
This work has been supported by a research grant from the Austrian Research
Agency FWF (Projektnummer P23902-N13).
\end{acknowledgement}

\begin{acknowledgement}
The author wishes to thank Hans Feichtinger and Karlheinz Gr\"{o}chenig for
valuable suggestions, criticism, and discussions about Gabor frames.
\end{acknowledgement}

\end{document}